\documentclass[12pt,reqno]{amsart}
\usepackage{parskip}
\usepackage{amsthm,amsmath,amssymb,oldgerm}
\usepackage[a4paper, margin=2.5cm]{geometry}
\usepackage{mathrsfs}
\usepackage{verbatim}
\usepackage[all]{xy}

\theoremstyle{plain}
\newtheorem{thm}{Theorem}[section]
\newtheorem{prop}[thm]{Proposition}
\newtheorem{cor}[thm]{Corollary}
\newtheorem{lem}[thm]{Lemma}

\numberwithin{equation}{section}

\let\H\relax
\let\L\relax
\let\O\relax
\let\dh\relax

\DeclareMathOperator{\dh}{{\it{d}}_{\mathrm{hyp}}}
\DeclareMathOperator{\cx} {\mathcal{C}_{{\it{X}}}^{{\it{k}}}}
\DeclareMathOperator{\bk} {\mathcal{B}_{{\it{X}}}^{2{\it{k}}}}
\DeclareMathOperator{\bkx} {\mathcal{B}_{\Omega_{\it{X}}}^{\it{k}}}
\DeclareMathOperator{\bkl} {\mathcal{B}_{\mathcal{L}}^{\it{k}}}
\DeclareMathOperator{\bkld} {\mathcal{B}_{\mathcal{L},{\it{D}}}^{\it{k}}}
\DeclareMathOperator{\hypx}{\mu_{X}^{\mathrm{hyp}}} 
\DeclareMathOperator{\muberkx}{\mu_{X}^{{\it{k}},\mathrm{ber}}} 
\DeclareMathOperator{\hypxd}{\mu_{\mathrm{X}^{\mathrm{d}}}^{\mathrm{hyp}}} 
\DeclareMathOperator{\hypxdvol}{\mu_{\mathrm{X}^{\mathrm{d}}, \mathrm{vol}}^{\mathrm{hyp}}}
\DeclareMathOperator{\H}{\mathbb{H}}
\DeclareMathOperator{\calH}{\mathcal{H}}

\DeclareMathOperator{\G}{\Gamma} 
\DeclareMathOperator{\L}{\mathcal{L}}

\DeclareMathOperator{\C}{\mathbb{C}}
\DeclareMathOperator{\Bk}{\mathcal{B}_{{\it{k}}}}
\DeclareMathOperator{\R}{\mathbb{R}} 
 
\DeclareMathOperator{\symxd}{\mathrm{Sym}^{\mathrm{d}}(\mathrm{X})} 
\DeclareMathOperator{\O}{\Omega_{{\it{X}}}} 
\DeclareMathOperator{\Ok}{\Omega_{{\it{X}}}^{\otimes{\it{k}}}} 
\DeclareMathOperator{\Stwo}{\mathcal{S}^{2}(\Gamma)} 
\DeclareMathOperator{\Sk}{\mathcal{S}^{2{\it{k}}}(\Gamma)}
\DeclareMathOperator{\rx}{{\it{r}}_{{\it{X}}}} 

\begin{document}

\title[Bergman kernel and symmetric product of Riemann surface]{Bergman kernel on
Riemann surfaces and K\"ahler metric on symmetric products}

\author[A. Aryasomayajula]{Anilatmaja Aryasomayajula}
\address{Indian Institute of Science Education and Research, Tirupati, 
Karakambadi Road, Mangalam (P.O.) Tirupati -517507,
Andhra Pradesh, India}

\email{anilatmaja@gmail.com}

\author[I. Biswas]{Indranil Biswas}

\address{School of Mathematics, Tata Institute of Fundamental Research,
Homi Bhabha Road, Mumbai 400005, India}

\email{indranil@math.tifr.res.in}

\subjclass[2010]{32N05, 32A25,11F03, 53C07}

\keywords{Bergman kernel; symmetric product; hyperbolic Riemann surface; Fubini-Study metric.}

\begin{abstract}
Let $X$ be a compact hyperbolic Riemann surface equipped with the Poincar\'e metric. For any
integer $k\, \geq\, 2$, we investigate the Bergman kernel associated to the
holomorphic Hermitian line bundle
$\Omega^{\otimes k}_X$, where $\O$ is the holomorphic cotangent bundle of $X$. Our first main
result estimates the corresponding Bergman metric on $X$ in terms of the Poincar\'e metric.
We then consider a certain natural embedding of the symmetric product of $X$ into a Grassmannian
parametrizing subspaces of fixed dimension of the space of all global holomorphic sections of
$\Omega^{\otimes k}_X$. The Fubini-Study metric on the Grassmannian restricts to a K\"ahler metric
on the symmetric product of $X$. The volume form for this restricted metric
on the symmetric product is estimated in terms
of the Bergman kernel of $\Omega^{\otimes k}_X$ and the volume form for the orbifold K\"ahler form
on the symmetric product given by the Poincar\'e metric on $X$. 
\end{abstract}

\maketitle

\section{Introduction}

Estimates of Bergman kernels associated to high tensor-powers of holomorphic line bundles defined on
compact complex manifolds has been 
studied since a long time. Tian \cite{Ti}, Zelditch \cite{Ze}, \cite{SZ}, Demailly \cite{De1}
\cite{De2}, and more recently, Ma and Marinescu in \cite{ma}, have done seminal works in 
this field. Recently, in \cite{au-ma2}, Auvray, Ma, and Marinescu have extended their estimates of Bergman kernels to noncompact 
hyperbolic Riemann orbi-surfaces.

In this article, we study derivatives of Bergman kernels associated to tensor powers of holomorphic line bundles. The first main result of the article, is an estimate of the 
Bergman metric, which is also an area of great interest in the field of analytic geometry and complex analysis.

Symmetric products of Riemann surfaces are extensively studied in algebraic geometry. The
symmetric products of a compact Riemann surface $X$ guide many of the algebraic geometric aspects
of $X$; see \cite{ACGH}, \cite{Ke}. These symmetric products also arise in mathematical physics because
they are important examples of vortex moduli spaces \cite{MN}, \cite{Ma}, \cite{Pe}, \cite{biswas1}, \cite{BR2}.
The topology of symmetric products of a Riemann surface was first studied by Macdonald in \cite{mcdon}. 

The second main result of the article, is an estimate of the volume form associated to a certain K\"ahler metric on the symmetric product of a compact hyperbolic Riemann surface.

We now describe the main results of the article. Let $X$ be a compact connected Riemann surface of genus at least two. It is equipped with the
Poincar\'e metric with constant scalar curvature $-1$. The holomorphic cotangent bundle of $X$,
which is denoted by $\O$, is very ample if $X$ is non-hyperelliptic. In general, $\O$ is a line bundle
with a connection of positive curvature induced by the Poincar\'e metric.

Take any positive integer $k$. The Bergman kernel for the line bundle $\Omega^{\otimes k}_X$ is given by the
sum of the point-wise norm-square of an orthonormal basis of $H^0(X,\, \Omega^{\otimes k}_X)$ for the $L^2$ Hermitian
metric on it; this Bergman kernel function is independent of the choice of the orthonormal basis. Using this
kernel function as the K\"ahler potential, a K\"ahler form on $X$ can be constructed, which is called the Bergman metric. The first main result of 
the article is an estimate of the Bergman metric on $X$, which is derived in terms of the 
Poincar\'e metric on $X$ and the Bergman kernel for $\Omega^{\otimes k}_X$. 

The following theorem describes our first main result, and is proved as Theorem \ref{thm1}.

{\bf{Main Theorem 1.}}\,
{\it Let $X$ be a compact hyperbolic Riemann surface $X$, and let $\hypx$ be the hyperbolic metric on $X$.
Let $\bkx$ denote the Bergman kernel associated to the holomorphic
cotangent bundle $\O$, and for any $z\in X$, let $\muberkx(z)$ denote the weight-$k$ Bergman metric on $X$, which is as expressed as 
$$
\muberkx(z):=-\frac{\sqrt{-1}}{2\pi}\partial\overline{\partial}\log\big(\bkx(z)\big)\,.
$$
Then, for any $k\geq 3$ and $z\in X$, we have the following estimate
\begin{align*}
&\bigg|\frac{\muberkx(z)}{\hypx( z)} \bigg|\,\leq\,
\frac{k^{2}}{\pi}\cdot\frac{\cx}{\bkx(z)}\cdot\bigg(\frac{4\cx}{\bkx(z)}+5+\frac{1}{2k}\bigg)+\frac{k}{2\pi},\\[0.1cm]
&\mathrm{where}\,\,\,\cx:=\frac{2k-1}{4\pi}\bigg(2+\frac{16}{\cosh^{2k-4}(\rx\slash 4)}+\frac{8}{\cosh^{2k-3}(\rx\slash 2)} \bigg)
 \notag \\[0.1cm]&+\frac{2k-1}{2\pi\sinh^{2}(\rx\slash4)}\cdot\bigg(\frac{1}{(2k-2)\cosh^{2k-3}(\rx\slash 2)}
+\frac{1}{(k-2)\cosh^{2k-4}(\rx\slash2)}\bigg)\,,
\end{align*}
and $\rx$ denotes the injectivity radius of $X$, which is as defined in equation \eqref{defnirad}.}

The symmetric product $\symxd$ mentioned earlier is the quotient of the Cartesian product
$X^d$ for the action of the permutation group on $X^d$
that permutes the factors in the Cartesian product. This $\symxd$ is a smooth complex projective variety of complex dimension
$d$. The symmetric product $\symxd$ has a canonical embedding into the Grassmann variety parametrizing the
subspace of $H^0(X,\, \Omega^{\otimes k}_X)$ of codimension $d$, for every positive $k$ sufficiently large.
We recall that this embedding sends an effective divisor $D$ on $X$ of degree $d$ to the point of the above Grassmannian that corresponds to image of the homomorphism
\begin{equation}\label{hom}
H^0(X,\, \Omega^{\otimes k}_X\otimes {\mathcal O}_X(-D))\, \longrightarrow\, H^0(X,\, \Omega^{\otimes k}_X)
\end{equation}
that occurs in the long exact sequence of cohomologies associated to the short exact sequence of
sheaves
$$
0\, \longrightarrow\, \Omega^{\otimes k}_X\otimes {\mathcal O}_X(-D)\, \longrightarrow\, \Omega^{\otimes k}_X
\, \longrightarrow\, \Omega^{\otimes k}_X\vert_D \, \longrightarrow\, 0
$$
on $X$. The homomorphism in \eqref{hom} can also be interpreted as the realization of the complex vector space
$H^0(X,\, \Omega^{\otimes k}_X\otimes {\mathcal O}_X(-D))$ as the subspace of
$H^0(X,\, \Omega^{\otimes k}_X)$ parametrizing all holomorphic sections of $\Omega^{\otimes k}_X$ that vanish
along the divisor $D$.

When $\symxd$ is embedded in the above Grassmannian, the Fubini-Study K\"ahler metric on the Grassmannian restricts to a K\"ahler metric on the
complex submanifold $\symxd$. We estimate the volume form of this K\"ahler metric on $\symxd$. This is done
in terms of the Bergman kernel for $\Omega^{\otimes k}_X$ and volume form on $\symxd$ associated to the orbifold K\"ahler
form on $\symxd$ given by the Poincar\'e metric on $X$.

The following theorem describes our second main result, and is proved as Theorem \ref{thm2}.

{\bf{Main Theorem 2.}}\,
{\it With notation as in Main Theorem 1, for any $d\geq 1$, let $\hypxdvol$ be the volume form associated to the metric induced by the hyperbolic metric on $\symxd$. Let $n_{k}:=(2k-1)(g-1)$ and $r_{k}:=n_{k}-d$, and let $\mathrm{Gr}(r_{k},\,n_{k})$ denote the Grassmannian parametrizing $r_{k}$-dimensional vector subspaces of $\C^{n_{k}}$. 
From homomorphism \eqref{hom}, we then have a holomorphic embedding 
$$
\varphi_{\Omega}^{k}\,:\, \symxd\,\hookrightarrow\, \mathrm{Gr}(r_{k},\,n_{k})\,.
$$
Let $\mu_{\symxd,\mathrm{vol}}^{\mathrm{FS},k}(z)$ denote the volume form, associated to the pull-back of the Fubini-Study metric on $\mathrm{Gr}(r_{k},\,n_{k})$, via the holomorphic embedding $\varphi_{\Omega}^{k}$. Then, for $k\,\gg\, 0$, and any $z\,:=\,(z_1,\,\cdots,\,z_d)\,\in\,\symxd$, we have the following estimate
\begin{align*}
\bigg|\frac{\mu_{\symxd,\mathrm{vol}}^{\mathrm{FS},k}(z)}{\hypxdvol(z)}\bigg|
\,\leq \,\prod_{i=1}^{d}\Bigg(\frac{k^{2}}{\pi}\cdot\frac{\cx}{\bkx(z_{i})}\cdot\bigg(\frac{4\cx}{\bkx(z_{i})}+5+\frac{1}{2k}\bigg)+\frac{k}{2\pi}\Bigg)+o_{z}(k)\, ,
\end{align*}
where $o_{z}(k)$ represents a smooth function on $\symxd$
for each $k$, which for any $z\in\symxd$, goes to zero as $k$ tends to infinity.}

A similar study of estimates of K\"ahler metrics on $\symxd$ was undertaken in \cite{uoh}, which we now describe for the benefit of the reader. Let $\eta(X)$ denote the gonality of $X$, which is the minimum of the degrees of nonconstant holomorphic maps from $X$ to ${\mathbb C}{\mathbb P}^1$. 

Then, for any $d \,<\, \eta(X)$, we have a holomorphic embedding
$$
\phi\,:\, \symxd \,\hookrightarrow \, {\rm Pic}^{d}(X)\, , \ \ z:\,=\,(z_1,\, \cdots,\,z_d)\,\longmapsto\,
\mathcal{O}_{X}(D)\, ,
$$
where $D\,:=z_1+\,\ldots+\,z_d$.

Let $\mu_{g}$ denote the flat Euclidean metric on the complex torus ${\rm Pic}^{d}(X)$. Then, $\mu_{X^{d}}^{\rm{can}}
\,:=\,\phi^{\ast}(\mu_{g})$ is a K\"ahler metric on $\symxd$. In \cite{uoh}, estimates of the volume form of
$\mu_{X^{d}}^{\rm{can}}$ on $\symxd$ were computed. Using these it was shown there that
any holomorphic automorphism of $\symxd$ is an isometry.

\section{Bergman kernels on Riemann surfaces}

\subsection{Hyperbolic Riemann surface}

Let $X$ be a compact hyperbolic Riemann surface of genus $g$, with $g\,>\,1$. The uniformization theorem says that
$X$ can be realized as a quotient space $\G\backslash\H$, where
$$
\H\,:=\,\lbrace z\,=\,x+\sqrt{-1}\cdot y\,\mid\, y=\mathrm{Im}(z)>0 \rbrace
$$
is the hyperbolic upper half-plane, and $\G\,\subset\,\mathrm{PSL}(2,\R)$ is a
torsion-free cocompact Fuchsian subgroup. Locally, identify
$X$ with its universal cover $\H$, and for brevity of notation, identify points on $X$ by the same letters as the points on
$\H$. Let $X_{\Gamma}$ denote a fixed fundamental domain of $X$.
 
Let
\begin{align}\label{defnhypmetric}
\mu^{\mathrm{hyp}}(z)\,:=\,\frac{\sqrt{-1}}{2}\cdot\frac{dz\wedge d\overline{z}}{y^{2}}=\frac{dxdy}{y^{2}}
\end{align} 
be the hyperbolic metric on $\H$, which is of constant negative curvature $-1$. This $\mu^{\mathrm{hyp}}$
induces a metric on $X$ because it is preserved under the action of $\mathrm{PSL}(2,\R)$. The corresponding metric on $X$ 
is compatible with the natural complex structure of $X$; this induced metric on $X$ will be denoted by $\hypx$. Locally, for 
$z\,=\,x+\sqrt{-1} y\in X$, the hyperbolic metric $\hypx(z)$ is given by the expression given in \eqref{defnhypmetric}.
 
Let $\dh(z,w)$ denote the natural distance function on $\H$, which is given by the hyperbolic metric $\mu^{\mathrm{hyp}}$.
Locally, identifying the Riemann surface $X$ with the hyperbolic-plane $\H$, for any $z,\,w\,\in\, X$, the geodesic distance between the points $z$ and $w$ on $X$ is denoted by $\dh(z,w)$. 

The injectivity radius of $X$ is given by the formula
\begin{align}\label{defnirad}
\rx\,:=\,\inf\big\lbrace \dh(z,\gamma z)\,\mid\, \,z\,\in\, \H,\ \gamma\,\in\,
\Gamma\backslash\lbrace\mathrm{Id}\rbrace\big\rbrace\, .
\end{align}

\subsection{Bergman kernel associated to a holomorphic line bundle}

Let $\L$ be a positive holomorphic line bundle on $X$ of degree $2(g-1)$ with a Hermitian metric $\|\cdot\|$. Let 
$H^{0}(X,\,\L)$ denote the complex vector space of global holomorphic sections of the line bundle $\L$. Let 
$$
n_{1}:\,=\,\dim H^{0}(X,\,\L)\,,
$$ 
and $n_{1}$ is equal to $g$ or $g-1$ depending on whether $\L$ is isomorphic to the holomorphic
cotangent bundle $\O$ or not. Similarly, for any $k\, \geq\, 2$, the dimension of $H^{0}(X,\,\L^{\otimes k})$,
the complex vector space of global holomorphic sections of the line bundle $\L^{\otimes k}$, is 
$$
n_{k}\,:=\,(2k-1)(g-1)\,.
$$
The Hermitian structure on $\L^{\otimes k}$ induced by the Hermitian structure on $\L$ will be denoted by $\|\cdot\|_{k}$.

We further assume that the curvature form $c_1(\L,\,\|\cdot\|)$ of the line bundle $\L$ satisfies the following condition:
\begin{align}\label{c1cond}
c_1(\L,\,\|\cdot\|)(z)\,:=\,-\frac{\sqrt{-1}}{2\pi}\partial\overline{\partial}\log\|s\|^{2}(z)\,=\,\frac{1}{2\pi}\hypx(z)
\end{align}
for every $z\,\in\, X$, where $s$ is any locally defined holomorphic section of $\L$. We note that this condition
determines $\|\cdot\|$ up to a positive constant scalar.

The K\"ahler metric $\hypx$ and the Hermitian structure $\|\cdot\|_{k}$ on $\L^{\otimes k}$ together
produce an $L^{2}$ inner-product on $H^{0}(X,\,\L^{\otimes k})$.
For any $k\,>\,0$, let $\lbrace s_{1},\,\cdots,\,s_{n_{k}}\rbrace$ denote 
an orthonormal basis of $H^{0}(X,\,\L^{\otimes k})$ with respect to this $L^{2}$ inner-product.

The Bergman kernel associated to the complex vector space $H^{0}(X,\,\L^{\otimes k})$ is given by the following formula
$$
\bkl(z)\,:=\,\sum_{i=1}^{n_{k}}\|s_{i}\|_{k}^{2}(z)\, .
$$

It should be clarified that the definition of the Bergman kernel $\bkl$ does not depend on the above choice of 
orthonormal basis for $H^{0}(X,\,\L^{\otimes k})$.

\subsection{Bergman kernel associated to the cotangent bundle}

Now specialize to the case of $\L\,=\,\O$, where $\O$ as before denotes the holomorphic cotangent line bundle of $X$.

For any $k\,>\,0$, let $\Sk$ denote the complex vector space of weight-$2k$ cusp forms on $X$. Locally, any $\omega\,
\in\, H^{0}(X,\,\O)$ can be realized as $f(z)dz$, where $f\,\in\,\Stwo$, which is based on the $\G$-invariance of the global form $f(z)dz$ on $\H$. Similarly, for any $k\,>\,0$, and $\omega\,\in\, H^{0}(X,\,\Ok)$, locally,
at any $z\,\in\, X$, we have $\omega(z)\,=\,f(z)dz^{\otimes k}$, where $f\in\Sk$ is a weight-$2k$ modular form with respect to $\Gamma$.

The point-wise metric function on $H^{0}(X,\,\Ok)$ is denoted by $\|\cdot\|_{\mathrm{hyp}}$. For any $\omega\,\in\,
H^{0}(X,\,\Ok)$, locally, at the point $z\,=\,x+\sqrt{-1} y\,\in\, X$, it is given by the following formula:
\begin{align*}
\|\omega\|_{\mathrm{hyp}}(z)\,=\,y^{k}|f(z)|\, ,
\end{align*}
where $\omega(z)\,=\,f(z)dz^{\otimes k}$; note that $\|\omega\|_{\mathrm{hyp}}$ is a real valued function on $X$, a 
fact which will come handy, later in the section. 

For any $\omega\,\in\, H^{0}(X,\,\O)$ and $z\,\in\, X$, it is straight-forward to check that
\begin{align*}
c_1(\O,\,\|\cdot\|_{\mathrm{hyp}})(z)\,=\,-\frac{\sqrt{-1}}{2\pi}\partial\overline{\partial}\log\|\omega\|_{\mathrm{hyp}}^{2}(z)
\,=\,\frac{1}{2\pi}\hypx(z)\, ,
\end{align*}
and that the cotangent bundle $\O$ satisfies the condition in \eqref{c1cond}.

The above point-wise metric $\|\cdot\|_{\mathrm{hyp}}$ induces an $L^{2}$-metric on $H^{0}(X,\,\Ok)$, which is denoted by
$\langle\cdot,\,\cdot\rangle_{\mathrm{hyp}}$.
For any pair $$\omega,\, \eta\,\in\, H^{0}(X,\,\Ok)\, ,$$
locally, at a point $z\,=\, x+\sqrt{-1} y\,\in\, X$, if we have 
$$
\omega(z)\,=\, f(z)dz^{\otimes k}\ \ \text{ and } \ \ \eta(z)\,=\,g(z)dz^{\otimes k}\, ,$$
then the $L^{2}$-metric is given by the
formula
\begin{align*}
\langle \omega,\eta\rangle_{\mathrm{hyp}}\,=\,\int_{X_{\Gamma}}y^{2k}f(z)\overline{g(z)}\hypx(z)\, ;
\end{align*}
recall that $X_{\Gamma}$ is a fundamental domain.

Let $\lbrace \omega_1,\,\cdots,\, \omega_{n_{k}}\rbrace$ be an orthonormal basis of 
$H^{0}(X,\,\Ok)$ with respect to the $L^{2}$-metric $\langle\cdot,\, 
\cdot\rangle_{\mathrm{hyp}}$. Recall that $n_k\,=\, (2k-1)(g-1)$ if $k\, \geq\, 2$
and $n_1\,=\, g$. Then, the Bergman kernel associated to line bundle $\Ok$ is given 
by the following formula
$$
\bkx(z)\,:=\,\sum_{i=1}^{n_{k}}\|\omega_i\|_{\rm{hyp}}^{2}(z)\, .
$$

From the above discussion, it is clear that $H^{0}(X,\,\Ok)\,\cong\,\Sk$ as complex vector spaces. So analyzing $\bk$, the Bergman 
kernel associated to the complex vector space $\Sk$, is equivalent to analyzing $\bkx$. We now define $\bk$, and also describe an 
infinite series representation of $\bk$. Our strategy to Main Theorem 1 and Main Theorem 2, is to analyze and estimate the infinite 
series representation of $\bk$.

For $f\,\in\,\Sk$, there is the following point-wise metric at $z\,=\,x+\sqrt{-1} y\in X$
\begin{align*}
\|f\|_{\mathrm{pet}}(z)\,:=\,y^{k}|f(z)|\, ,
\end{align*}
which is also known as the Petersson norm. The Petersson norm induces an $L^{2}$-metric on $\Sk$, which is also known as the Petersson inner product. For any $f,\,h\,\in\,\Sk$, the Petersson inner-product is given by the formula
\begin{align*}
\langle f,\, h\rangle_{\mathrm{pet}}\,:=\,\int_{X_{\Gamma}}y^{2k}f(z)\overline{h(z)}\hypx(z)\, .
\end{align*}

Let $\lbrace f_1,\,\cdots,\,f_{n_{k}}\rbrace $ be an orthonormal basis for $\Sk$ with respect to the Petersson 
inner product. Then, for any $z\,=\,x+\sqrt{-1} y \,\in \,X$, the Bergman kernel associated to the complex vector 
space $\Sk$ is given by the formula
$$
\bk(z)\,:=\,\sum_{i=1}^{n_{k}}\|f_{i}\|_{\mathrm{pet}}^{2}(z)\,=\,y^{2k}\sum_{i=1}^{n_{k}}|f_{i}|^{2}(z)\, .
$$

The Bergman kernel $\bk(z)$ can also be defined by the infinite series (see Proposition $1.3$ on p. $77$ in \cite{frietag})
\begin{equation}\label{seriesbergmanreln0}
\bk(z)\,=\, \frac{(2k-1)(2\sqrt{-1}y)^{2k}}{4\pi}\sum_{\gamma\in\Gamma}\frac{1}{( z-\gamma\overline{z})^{2k}}
\cdot\frac{1}{(c\overline{z}+d)^{2k}}\, ,
\end{equation}
where $\gamma\,=\,\left(\begin{array}{cc} a&b\\c&d\end{array}\right)\, \in\, \Gamma$. The expression for the Bergman kernel $\bk(z)$ given in \cite{frietag} is missing a factor of $(2\sqrt{-1})^{2k}$, which is taken into account in the above formula.

Since $H^{0}(X,\,\Ok)\,\cong\,\Sk$ as complex vector spaces, locally, we have the following relation of Bergman kernels
$$
\bkx(z)\,=\, \bk(z)\, .
$$

\subsection{Estimates of Bergman kernel}

We now describe some asymptotic estimates of the Bergman kernel associated to holomorphic line 
bundles, which satisfy equation \eqref{c1cond}.

Let $\L$ be a holomorphic Hermitian line bundle as above, satisfying equation \eqref{c1cond}, and let $\mathcal{V}$ be a 
holomorphic vector bundle of rank $r$ on $X$. Let $\mathcal{B}_{\mathcal{V}\otimes \mathcal{L}}^{k}$ denote the Bergman kernel associated to the holomorphic vector bundle $\mathcal{V}\otimes \mathcal{L}^{\otimes k}$. Then, for $k\,\gg\, 0$, and for any $z\,\in\, X$, combining results from \cite{ma} (Theorems 6.1.1 and 6.2.3) with equation \eqref{c1cond}, 
it follows that
\begin{align}\label{bkest0}
\frac{1}{k} \mathcal{B}_{\mathcal{V}\otimes\mathcal{L}}^{k}(z)\,=\,r\cdot \bkl(z)+O\big(k^{-1}\big)
\,=\,\frac{r}{2\pi}+ O\big(k^{-1}\big)\,.
\end{align}
From equation \eqref{bkest0} it follows that for large values of $k$ the value of $\bkl(z)$
depends asymptotically only on $k$.

We now describe estimates of $\bkx$, which are derived using arguments from \cite{am}. In \cite{am2}, extending the arguments from \cite{am}, an explicit estimate of $\bk(z)$ is derived, which is stated below. For a compact hyperbolic Riemann surface $X$, and for any $k\,\geq\, 3$, and $z:=\,x+iy\,\in\, X$, substituting $\delta=0$ in estimate $(5)$ in Main theorem from \cite{am2}, we have the following estimate 
\begin{align}
&\bkx(z)\,\leq\, \frac{(2k-1)(2y)^{2k}}{4\pi}\sum_{\gamma\in\Gamma}\frac{1}{\big| z-\gamma\overline{z}\big|^{2k}
\cdot\big|c\overline{z}+d\big|^{2k}}\notag\\[0.1cm]
&=\, \frac{(2k-1)(2y)^{2k}}{4\pi}\sum_{\gamma\in\Gamma}\frac{1}{\cosh^{2k}(\dh(z,\gamma z)\slash 2)}
\,\leq\, \cx\label{estimatebkx} 
\end{align}
where
\begin{align}
& \cx\,:=\,
\frac{2k-1}{4\pi}\bigg(2+\frac{16}{\cosh^{2k-4}(\rx\slash 4)}+\frac{8}{\cosh^{2k-3}(\rx\slash 2)} \bigg)
 \notag \\[0.1cm]&+\,\frac{2k-1}{2\pi\sinh^{2}(\rx\slash4)}\cdot\bigg(\frac{1}{(2k-2)\cosh^{2k-3}(\rx\slash 2)}
+\frac{1}{(k-2)\cosh^{2k-4}(\rx\slash2)}\bigg).\label{defncx}
 \end{align}

\subsection{Bergman metric on $X$}

For any $k\,>\,0$ and $z\,\in\, X$, the weight-$k$ Bergman metric on $X$ is expressed as
\begin{equation}\label{eqbb}
\muberkx(z)\,:=\,-\frac{\sqrt{-1}}{2\pi}\partial\overline{\partial}\log(\bkx(z))\, .
\end{equation}

We now estimate the ratio $\muberkx\slash \hypx$; the following function is introduced in order
to facilitate the computations that
ensue. For any $k\,>\,0$ and $z\,\in\, \H$, put
\begin{align}\label{defnB(z)}
\Bk(z)\,:=\,\sum_{\gamma\in\Gamma}\frac{(\sqrt{-1})^{2k}}{( z-\gamma\overline{z})^{2k}\cdot (c\overline{z}+d)^{2k}}\, .
\end{align}

Combining equations \eqref{seriesbergmanreln0} and \eqref{defnB(z)}, we arrive at the equation
\begin{align}\label{bkxbzreln}
\bkx(z)\,=\,\frac{(2k-1)}{4\pi}\cdot (2y)^{2k}\Bk(z)\, .
\end{align}
From the above equation it follows that $\Bk(z)$ is a real-valued positive function. So, we make the following observation
\begin{align}\label{logbkxbzreln}
\log (\bkx(z))\,=\,\log\left(\frac{4^{k}(2k-1)}{4\pi}\right)+\log (y^{2k})+\log (\Bk(z))\, .
\end{align}
Furthermore, for any $z\,=\,x+\sqrt{-1} y$ and $w\,=\, u+\sqrt{-1} v\,\in\, \H$, the following holds:
\begin{align*}
\cosh^{2}(\dh(z,w)\slash 2)\,=\,\frac{\big|z-\overline{w}\big|^{2}}{4yv}\, .
\end{align*}

For any $n \,\geq\, k\,>\,0 $ and $z\,\in\, X$, using the above formula, and from the
estimate in \eqref{estimatebkx}, we have
\begin{align}\label{estimateB(z)}
\big|\mathcal{B}_{n}(z)\big|\,\leq\, \sum_{\gamma\in\Gamma}\frac{1}{\big| z-\gamma\overline{z}\big|^{2n}\cdot\big|c\overline{z}+d\big|^{2n}}
\,=\,\frac{1}{(2y)^{2n}}
\cdot\sum_{\gamma\in\Gamma}\frac{1}{\cosh^{2n}(\dh(z,\gamma z)\slash 2)}\notag \\[0.1em]\leq\, \frac{1}{(2y)^{2n}}
\cdot\sum_{\gamma\in\Gamma}\frac{1}{\cosh^{2k}(\dh(z,\gamma z)\slash 2)}\,\leq\, \frac{4\pi}{(2k-1)}\cdot\frac{\cx}{(2y)^{2n}}\, .
\end{align}

\begin{prop}\label{prop1}
With notation as above, for any $k\,>\,0$ and $z\,\in\, X$,
\begin{align}\label{prop1eqn}
\muberkx(z)\,=\,\frac{k}{2\pi}\hypx(z)+\frac{y^{2}}{\pi}\Bigg(\displaystyle\frac{\displaystyle\frac{\partial \Bk(z)}{\partial z}\frac{\partial \Bk(z)}{\partial \overline{z}}}{\Bk^{2}(z)}-\displaystyle\frac{\displaystyle\frac{\partial^{2}\Bk(z)}{\partial z\partial \overline{z}}}{\Bk(z)}\Bigg)\hypx(z)\, .
\end{align}
\end{prop}

\begin{proof}
For any $k\,>\,0$ and $z\,\in\, X$, from equation \eqref{logbkxbzreln},
\begin{align}\label{prop1eqn1}
\muberkx(z)\,=\,-\frac{\sqrt{-1}}{2\pi}\partial\overline{\partial}\log (\bkx(z))
\,=\,-\frac{\sqrt{-1}}{2\pi}\partial\overline{\partial}\log (y^{2k}) -\frac{\sqrt{-1}}{2\pi}\partial\overline{\partial}\log(\Bk(z))\, .
\end{align}
Now observe that
\begin{align}\label{prop1eqn2}
-\frac{\sqrt{-1}}{2\pi}\partial\overline{\partial}\log (y^{2k})\,=\,\frac{k}{2\pi}\hypx(z)\, .
\end{align} 

Using equation \eqref{defnhypmetric}, we compute
\begin{align}\label{prop1eqn3}
 -\frac{\sqrt{-1}}{2\pi}\partial\overline{\partial}\log(\Bk(z))\,=
&\frac{\sqrt{-1}}{2\pi}\Bigg(\displaystyle\frac{\displaystyle\frac{\partial \Bk(z)}{\partial z}\frac{\partial
\Bk(z)}{\partial \overline{z}}}{\Bk^{2}(z)}-\displaystyle\frac{\displaystyle\frac{\partial^{2}\Bk(z)}{\partial z\partial
\overline{z}}}{\Bk(z)}\Bigg)dz\wedge d\overline{z}\notag\\[0.1em]=&\frac{y^{2}}{\pi}\Bigg(\displaystyle\frac{\displaystyle\frac{\partial \Bk(z)}{\partial z}\frac{\partial \Bk(z)}{\partial \overline{z}}}{\Bk^{2}(z)}-\displaystyle\frac{\displaystyle\frac{\partial^{2}\Bk(z)}{\partial z\partial \overline{z}}}{\Bk(z)}\Bigg)\hypx(z).
\end{align}
Combining equations \eqref{prop1eqn1}, \eqref{prop1eqn2} and \eqref{prop1eqn3} the proof of the proposition is completed. 
\end{proof}

We now estimate each of the terms involved in the right-hand side of equation \eqref{prop1eqn}. 

\begin{lem}\label{lem1}
With notation as above, for any $k\,\geq\, 3$ and $z\,=\,x+ \sqrt{-1} y\,\in\, \H$, the following estimate holds:
\begin{align*}
\bigg|\frac{\partial \Bk(z)}{\partial z} \bigg|\,\leq\, \frac{16\pi k}{2k-1}\cdot\frac{\cx}{(2y)^{2k+1}}\, ,
\end{align*}
where the constant $\cx$ is defined in \eqref{defncx}.
\end{lem}

\begin{proof}
For any $k\,\geq\, 3$ and $z\,=\,x+\sqrt{-1} y\,\in\, \H$, from equation \eqref{defnB(z)}, we have
\begin{align*}
\frac{\partial \Bk(z)}{\partial z} \,=\,-\sum_{\gamma\in\G}\frac{2k\cdot (\sqrt{-1})^{2k}}{(z-\gamma \overline{z})^{2k+1}\cdot (c\overline{z}+d)^{2k}}\, ,
\end{align*}
which implies that
\begin{align*}
\bigg|\frac{\partial \Bk(z)}{\partial z} \bigg|\,\leq\,
\sum_{\gamma\in\G}\frac{2k}{\big|z-\gamma \overline{z}\big|^{2k+1}\cdot \big|c\overline{z}+d\big|^{2k}}\, .
\end{align*}

For any $\gamma\,\in\,\G$, let $\gamma z\,=\,x_{\gamma}+\sqrt{-1}\cdot y_{\gamma}\,\in\, \H$. Then observe that
\begin{align}\label{lem1eqn1}
\big|z-\gamma\overline{z}\big|\,=\,\sqrt{(x-x_{\gamma})^{2}+(y+y_{\gamma})^{2}},
\,\,\,\,\mathrm{hence}\,\,\frac{1}{\big|z-\gamma\overline{z}\big|}\,\leq\, \frac{1}{y}\, .
\end{align}

Combining the above observation with estimate \eqref{estimatebkx}, we derive that
\begin{align*}
\bigg|\frac{\partial \Bk(z)}{\partial z}\bigg|\,\leq\,
\frac{2k}{y}\sum_{\gamma\in\G}\frac{1}{\big|z-\gamma \overline{z}\big|^{2k}\cdot \big|c\overline{z}+d\big|^{2k}}\\[1em]
\leq\, \frac{2k}{y}\sum_{\gamma\in\G}\frac{1}{\big|z-\gamma \overline{z}\big|^{2k}\cdot \big|c\overline{z}+d\big|^{2k}}\,\leq\, \frac{2k}{y}
\cdot\frac{4\pi \cx}{(2k-1)(2y)^{2k}},
\end{align*}
which completes the proof.
\end{proof}

\begin{lem}\label{lem2}
With notation as above, for any $k\,\geq\, 3$ and $z\,=\,x+\sqrt{-1} y\,\in\, \H$, the following estimate holds:
\begin{align*}
\bigg|\frac{\partial \Bk(z)}{\partial \overline{z}} \bigg|\,\leq\, \frac{16\pi k}{2k-1}\cdot\frac{\cx}{(2y)^{2k+1}}\, . 
\end{align*}
\end{lem}

\begin{proof}
For any $k\,>\,0$ and $z\,=\,x+\sqrt{-1} y\,\in\, \H$, as the function $\Bk(z)$ is real valued, we have
\begin{align*}
\Bk(z)\,=\,\sum_{\gamma\in\G}\frac{(\sqrt{-1})^{2k}}{(z-\gamma \overline{z})^{2k}\cdot (c\overline{z}+d)^{2k}}\,=
\,\sum_{\gamma\in\G}\frac{(\sqrt{-1})^{2k}}{(\overline{z}-\gamma z)^{2k}\cdot (c z+d)^{2k}}\, . 
\end{align*}
So for any $k\,\geq\, 3$ and $z\,=\,x+\sqrt{-1} y\,\in\, \H$, we have
\begin{align}\label{lem2eqn1}
\frac{\partial \Bk(z)}{\partial \overline{z}} \,=\,
-\sum_{\gamma\in\G}\frac{2k\cdot (\sqrt{-1})^{2k}}{(\overline{z}-\gamma z)^{2k+1}\cdot (cz+d)^{2k}}\, .
\end{align}
The lemma now follows from same arguments as in Lemma \ref{lem1}. 
\end{proof}

\begin{lem}\label{lem3}
For any $k\,\geq\, 3$ and $z\,=\,x+\sqrt{-1} y\,\in\, \H$, the following estimate holds:
\begin{align*}
\bigg|\frac{\partial^{2} \Bk(z)}{\partial z\partial \overline{z}} \bigg|\,\leq\,
\frac{\pi(80 k^{2}+8k)}{(2k-1)}\cdot \frac{\cx}{(2y)^{2k+2}}\, .
\end{align*}
\end{lem}

\begin{proof}
For any $\gamma\,=\,\left(\begin{array}{cc} a&b\\c&d\end{array}\right)\in\Gamma $ and $z\in\H$, observe that
$$
\frac{\partial (\gamma z)}{\partial z}\,=\,\frac{1}{(cz+d)^{2}}\, . 
$$
Combining the above equation with equation \eqref{lem2eqn1}, we derive
\begin{align}\label{lem3eqn1}
\frac{\partial^{2} \Bk(z)}{\partial z\partial \overline{z}} \,=\,
-\frac{\partial}{\partial z}\Bigg(\sum_{\gamma\in\G}\frac{2k\cdot (\sqrt{-1})^{2k}}{(\overline{z}-\gamma z)^{2k+1}\cdot (cz+d)^{2k}}\Bigg)\notag
\\[1em]=\, -\sum_{\gamma\in\G}\frac{2k(2k+1)\cdot (\sqrt{-1})^{2k}}{(\overline{z}-\gamma z)^{2k+2}\cdot (cz+d)^{2k+2}}+
\sum_{\gamma\in\G}\frac{4ck^{2}\cdot (\sqrt{-1})^{2k}}{(\overline{z}-\gamma z)^{2k+1}\cdot (cz+d)^{2k+1}}.
\end{align}

We now estimate each of the terms on the right-hand side of the above equation. From estimate \eqref{estimateB(z)} we have the following estimate for the first term on the right-hand side of equation \eqref{lem3eqn1}
\begin{align}\label{lem3eqn2}
-\sum_{\gamma\in\G}\frac{2k(2k+1)\cdot (\sqrt{-1})^{2k}}{(\overline{z}-\gamma z)^{2k+2}\cdot (cz+d)^{2k+2}}
\,\leq\, 2k(2k+1)\mathcal{B}_{k+1}(z)
\,\leq\, \frac{8\pi k (2k+1)}{(2k-1)}\cdot \frac{\cx}{(2y)^{2k+2}}\, .
\end{align}

For any $\gamma=\left(\begin{array}{cc} a&b\\c&d\end{array}\right)\in\Gamma $ and $z=x+\sqrt{-1} y\in\H$, observe that
\begin{align}\label{lem3eqn3}
\frac{c}{ \big|cz+d\big|}\,=\,\frac{c}{\sqrt{(cx+d)^{2}+(cy)^{2}}}\,\leq\, \frac{1}{y}\, . 
\end{align}

Combining inequalities \eqref{lem1eqn1} and \eqref{lem3eqn3} with the estimate in \eqref{estimateB(z)}, we have the following 
estimate for the second term on the right-hand side of equation \eqref{lem3eqn1}
\begin{align}\label{lem3eqn4}
&\sum_{\gamma\in\G}\frac{4ck^{2}\cdot (\sqrt{-1})^{2k}}{(\overline{z}-\gamma z)^{2k+1}\cdot (cz+d)^{2k+1}}
\notag\\[1em]&\leq\, \sum_{\gamma\in\G}\frac{4ck^{2}}{\big|cz+d\big|}\cdot\frac{1}{\big|\overline{z}-
\gamma z\big|}\cdot\frac{1}{\big|\overline{z}-\gamma z\big|^{2k}\cdot \big|cz+d\big|^{2k}}
\,\leq\, \frac{64\pi k^{2}}{2k-1}\cdot\frac{\cx}{(2y)^{2k+2}}\, .
\end{align}
Combining estimates \eqref{lem3eqn2} and \eqref{lem3eqn4} with \eqref{lem3eqn1}
we arrive at the following estimate:
\begin{align*}
\frac{\partial^{2} \Bk(z)}{\partial z\partial \overline{z}}\leq \frac{8\pi k (2k+1)}{(2k-1)}\cdot
\frac{\cx}{(2y)^{2k+2}}+\frac{64\pi k^{2}}{2k-1}\cdot\frac{\cx}{(2y)^{2k+2}}\,\leq\, 
\frac{\pi(80 k^{2}+8k)}{(2k-1)}\cdot \frac{\cx}{(2y)^{2k+2}}
\end{align*}
which completes the proof of the lemma.
\end{proof}

\begin{thm}\label{thm1}
With notation as above, for any $k\,\geq\, 3$ and $z\,\in\, X$, the following estimate holds:
\begin{align*}
\bigg|\frac{\muberkx(z)}{\hypx( z)} \bigg|\,\leq\,
\frac{k^{2}}{\pi}\cdot\frac{\cx}{\bkx(z)}\cdot\bigg(\frac{4\cx}{\bkx(z)}+5+\frac{1}{2k}\bigg)+\frac{k}{2\pi}\, ,
\end{align*}
where $\muberkx(z)$ is defined in \eqref{eqbb}.
\end{thm}

\begin{proof}
For any $k\,\geq\, 3$ and $z\,\in\, X$, from Proposition \ref{prop1} we have
\begin{align}\label{thm1eqn1}
\bigg|\frac{\muberkx(z)}{\hypx( z)} \bigg|\leq\frac{k}{2\pi}+\frac{y^{2}}{\pi}\Bigg|\displaystyle\frac{\displaystyle\frac{\partial \Bk(z)}{\partial z}\frac{\partial \Bk(z)}{\partial \overline{z}}}{\Bk^{2}(z)}\Bigg|+\frac{y^{2}}{\pi}\Bigg|\displaystyle\frac{\displaystyle\frac{\partial^{2}\Bk(z)}{\partial z\partial \overline{z}}}{\Bk(z)}\Bigg|\, .
\end{align}
Combining Lemma \ref{lem1}, Lemma \ref{lem2}, and Lemma \ref{lem3} with \eqref{bkxbzreln}, we arrive at the following estimate\\[0.05cm]
\begin{align}\label{thm1eqn2}
&\frac{y^{2}}{\pi}\Bigg|\displaystyle\frac{\displaystyle\frac{\partial \Bk(z)}{\partial z}
\frac{\partial \Bk(z)}{\partial \overline{z}}}{\Bk^{2}(z)}\Bigg|+
\frac{y^{2}}{\pi}\Bigg|\displaystyle\frac{\displaystyle\frac{\partial^{2}\Bk(z)}{\partial z
\partial \overline{z}}}{\Bk(z)}\Bigg|\notag\\[1em]
 &\leq\, \frac{y^{2}}{\pi}\cdot\Bigg(\frac{4k\cx}{\frac{2k-1}{4\pi}\cdot(2y)^{2k+1}\Bk(z)}\Bigg)^{2}+
\frac{y^{2}}{\pi}\cdot\frac{(20k^{2}+2k)\cx}{\frac{2k-1}{4\pi}\cdot(2y)^{2k+2}\Bk(z)} \notag\\[1em]
&\leq\, \frac{k^{2}}{\pi\cdot}\frac{\cx}{\bkx(z)}\cdot\bigg(\frac{4\cx}{\bkx(z)}+5+\frac{1}{2k}\bigg).
\end{align}
Combining estimates \eqref{thm1eqn1} and \eqref{thm1eqn2} completes the proof of the theorem. 
\end{proof}

\begin{cor}\label{cor1}
With notation as above, for any $z\,\in\, X$, the following estimate holds:
\begin{align*}
\lim_{k\rightarrow\infty}\frac{1}{k^{2}}\bigg|\frac{\muberkx(z)}{\hypx( z)} \bigg|\,\leq\, \frac{26}{\pi}\, .
\end{align*}
\end{cor}

\begin{proof}
For $k\,\gg\, 0$, from the estimates in \eqref{defncx} and \eqref{bkest0} we have the estimates, 
\begin{align*}
&\cx\leq\frac{2k-1}{2\pi}+ O\bigg(\frac{k}{\cosh^{2k-4} (\rx\slash 2)}\bigg)\\
&{\rm and} \ \ \frac{1}{k}\bkx(z)= \frac{1}{2\pi}+ O\big(k^{-1}\big)\, ,
\end{align*}
respectively. Using the above estimates, and the fact that both the above estimates are uniform in $k$, we derive that
\begin{align*}
\frac{\cx}{\bkx(z)}\,\leq\, 2+O\big(k^{-1}\big)\, ,
\end{align*}
and the above estimate is uniform in $k$. Combining the above inequality with Theorem \ref{thm1} completes the proof.
\end{proof}

\section{Comparison of metrics on symmetric product of a Riemann surface}

\subsection{Symmetric product of a Riemann surface} 

Let $X^d\,:=\,X\times\cdots\times X$ denote the $d$-fold Cartesian product of $X$. For each $1\,\leq \,i\,\leq \,d$, let 
$$
p_{i}\,:\,X^{d}\,\longrightarrow\, X
$$
denote the projection to the $i$-th factor. Put
$$
\hypxd\,:=\,\sum_{i=1}^{d}p_{i}^{\ast}\hypx\, ,
$$
which defines a K\"ahler metric on $X^d$. 

Let $S_d$ denote the group of permutations of the set $\{1,\, \cdots,\, d\}$. It acts on $X^d$ by 
permuting the factors of the Cartesian product. The resulting quotient $X^d/S_d$ is called
the $d$-fold symmetric product of $X$, and it is denoted by $\symxd$. This $\symxd$ is
an irreducible smooth complex projective variety of complex dimension $d$. 
The above K\"ahler metric $\hypxd$ on $X^d$ is evidently
invariant under the action of $S_d$. Therefore, $\hypxd$ produces an orbifold K\"ahler metric on
$\symxd$. This orbifold K\"ahler metric on $\symxd$ given by $\hypxd$ will again be
denoted by $\hypxd$, for brevity of notation.

The volume form on $\symxd$ associated to $\hypxd$ will be denoted
by $\hypxdvol$.

As before, let $\L$ be a holomorphic Hermitian
line bundle over $X$ of degree $2(g-1)$, satisfying equation \eqref{c1cond}.
Recall that the vector space $H^{0}(X,\,\L^{\otimes k})$ is equipped with an $L^2$ inner product.
Choose an orthonormal basis of $H^{0}(X,\,\L^{\otimes k})$. Using it,
the complex Hermitian space $ H^{0}(X,\,\L^{\otimes k})$ gets identified with
$\C^{n_{k}}$ equipped with the standard Euclidean metric, where $n_{k}\,=
\,\dim H^{0}(X,\,\L^{\otimes k})\,= \,(2k-1)(g-1)$.

Set
$$
r_{k}\,=\, n_{k}-d\,.
$$
Let $\mathrm{Gr}(r_{k},\,n_{k})$ denote the Grassmannian parametrizing $r_{k}$-dimensional vector subspaces of $\C^{n_{k}}$. For
any $z:\,=\, (z_1,\, \cdots,\, z_d)\, \in\, \symxd$, let $D:\,=\,x_{1}+\,\ldots\,+x_{d}$ be the divisor associated to the point $z$ on $X$, where $x_{1}\,=\, z_{1},\,\cdots,\,x_{d}\,=\,z_{d}$. Consider
the injective homomorphism
\begin{equation}\label{gri}
H^0(X,\, \L^{\otimes k}\otimes {\mathcal O}_{X}(-D))
\,\hookrightarrow\, H^0(X,\, \L^{\otimes k})
\end{equation}
given by the subsheaf
$\L^{\otimes k}\otimes {\mathcal O}_{X}(-D)\, \subset\,
\L^{\otimes k}$.
For any $k$ with $2(k-1)(g-1)\, >\, d$, we have
$$
\dim H^0(X,\, \L^{\otimes k}\otimes {\mathcal O}_{X}(-D))\,=\, (2k-1)(g-1)-d
$$
because $H^1(X,\, \L^{\otimes k}\otimes {\mathcal O}_{X}(-D))\,=\,0$ (Serre
duality) and Riemann--Roch. Therefore, the subspace in \eqref{gri} gives an element of
$\mathrm{Gr}(r_{k},\,n_{k})$.

Consequently, we have a map
\begin{equation}\label{mvp}
\varphi_{\L}^{k}\,:\, \symxd\,\longrightarrow\, \mathrm{Gr}(r_{k},\,n_{k})
\end{equation}
$$
(z_{1},\,\cdots,\, z_{d})\,\longmapsto\,
H^{0}(X,\,\L^{\otimes k}\otimes {\mathcal O}_{X}(-D))\,\subset\, 
H^{0}(X,\,\L^{\otimes k})\, .
$$
It is known that $\varphi_{\L}^{k}$ is holomorphic embedding (cf. \cite{biswas1}).

\subsection{Grassmannian and the Fubini-Study metric}

We now describe the Fubini-Study metric defined over a Grassmannian. We refer the reader to section five of chapter one of 
\cite{griffiths}, and chapter four of \cite{ballman}, for standard material on Grassmannians and Fubini-Study metrics on them.

As before, $r_{k}\,=\, n_{k}-d$. Let
$$\pi\,:\,\calH\,\longrightarrow\, \mathrm{Gr}(r_{k},\,n_{k})$$ 
denote the tautological vector bundle of rank $r_{k}$ whose fiber over any point $p\,\in\, \mathrm{Gr}(r_{k},\,n_{k})$ is the $r_{k}$ dimensional subspace of $\C^{n_{k}}$ represented by the point $p$. The vector bundle $\calH$ is a sub-bundle of the trivial vector bundle $\mathrm{Gr}(r_{k},\,n_{k})\times \C^{n_{k}}\,\longrightarrow\, \mathrm{Gr}(r_{k},\,n_{k})$. We now describe a local chart around a fixed point $p_{0}\,\in\, \mathrm{Gr}(r_{k},\,n_{k})$, and a holomorphic frame for the
vector bundle $\calH$ over this chart. 

Let $E\,:=\,\lbrace e_{1},\,\ldots ,\,e_{n_{k}}\rbrace$ denote the standard basis of $\C^{n_{k}}$, and without loss of
generality, let the vector subspace corresponding to the point $p_{0}\,\in\, \mathrm{Gr}(r_{k},\,n_{k})$ be spanned
by the vectors $\lbrace e_{1},\,\cdots ,\,e_{r_{k}}\rbrace$. Take $N_{k}\,=\, r_{k}(n_{k}-r_{k})$. Identify $\C^{N_{k}}:=\C^{r_{k}(n_{k}-r_{k})}$
with the matrices of the form
$$
M_{z}\,:=\,\left(\begin{matrix}
z_{11}&\cdots& z_{1r_{k}}\\ \vdots &\cdots &\vdots\\ z_{n_{k}-r_{k},1}&\cdots & z_{n_{k}-r_{k},r_{k}}
\end{matrix}
\right)
$$
with complex entries.

Now to each $z\in \C^{N_{k}}$, we associate the point $p_z$ on $\mathrm{Gr}(r_{k},\,n_{k})$, which denotes the vector space spanned by the columns of the matrix
$$
\widetilde{M}_{z}\,:=\,
\left( \begin{matrix}\mathrm{Id}_{r_{k}}\\ M_{z} \end{matrix}\right)\, ,
$$
where $\mathrm{Id}_{r_{k}}$ is the identity matrix of size $r_{k}\times r_{k}$.

The identification of points on $\C^{N_{k}}$ with subspaces of $\C^{n_{k}}$ of dimension $r_{k}$ induces an injective
map from $\C^{N_{k}}$ onto a Zariski open dense subset $U\,\subset\,\mathrm{Gr}(r_{k},\,n_{k})$ with $p_0\, \in\, U$, which will be our local chart around the point $p_{0}$.

With notation as above, the columns of the above matrices $\widetilde{M}_{z}$ give us a
local holomorphic frame for the vector bundle $\calH$ over $U$. In other words,
$\calH$ is trivialized over $U$. As noted before, $\calH$ is a sub-bundle of the trivial vector bundle $\mathrm{Gr}(r_{k},\,n_{k})\times {\mathbb C}^{n_{k}}$. Considering the restriction of the standard inner product on ${\mathbb C}^{n_{k}}$ 
to the fibers of $\calH$ we get a Hermitian structure on $\calH$; this Hermitian structure on $\calH$ is denoted by $\| \cdot \|_{\calH}$. So for any $\mathfrak{s}\,\in\, H^{0}(U,\, \calH)$ given by an element of ${\mathbb C}^{r_{k}}$ (using the
above trivialization of $\calH\vert_U$), and for any $p_{z}\,\in\, U\, \subset\, \mathrm{Gr}(r_{k},\,n_{k})$, we have the following expression for $\| \mathfrak{s} \|_{\calH}^{2}(p_{z})$: 
\begin{align}\label{H-metric1}
\| \mathfrak{s} \|_{\calH}^{2}(p_{z})\,:=\,\mathfrak{s}^{\ast}\widetilde{M}_{z}^{\ast}\widetilde{M}_{z}\mathfrak{s}
\,=\,\mathfrak{s}^{\ast}(\mathrm{Id}_{r_{k}}+M_{z}^{\ast}M_{z})\mathfrak{s}\, , 
\end{align}
where $\mathfrak{s}^{\ast}$ denotes the conjugate transpose of $\mathfrak{s}$, and the section $\mathfrak{s}$ is viewed as a vector in $\C^{r_{k}}$. 

The holomorphic tangent bundle $T\mathrm{Gr}(r_{k},\,n_{k})$ is holomorphically identified with the vector bundle
$$
\text{Hom}(\calH,\, (\mathrm{Gr}(r_{k},\,n_{k})\times {\mathbb C}^{n_{k}})/\calH)\,=\, 
((\mathrm{Gr}(r_{k},\,n_{k})\times {\mathbb C}^{n_{k}})/\calH)\otimes {\calH}^*
$$
(recall that ${\calH}$ is a sub-bundle of the trivial vector bundle $\mathrm{Gr}(r_{k},\,n_{k})\times {\mathbb C}^{n_{k}}$). 
Just as the standard inner product on ${\mathbb C}^{n_{k}}$ produces a Hermitian structure on $\calH$, it also
produces a Hermitian structure on the quotient $(\mathrm{Gr}(r_{k},\,n_{k})\times {\mathbb C}^{n_{k}})/\calH$ by identifying the quotient with the orthogonal complement $\calH^\perp$ and then restricting the standard inner product to this orthogonal complement. Therefore, we get a Hermitian structure on
$$
T\mathrm{Gr}(r_{k},\,n_{k})\,=\, \text{Hom}(\calH,\, (\mathrm{Gr}(r_{k},\,n_{k})\times {\mathbb C}^{n_{k}})/\calH)\, .
$$
This Hermitian form on $\mathrm{Gr}(r_{k},\,n_{k})$ is in fact K\"ahler; it is known as the Fubini-Study metric, and it
is denoted by $\mu_{\mathrm{Gr}}^{\mathrm{FS}}$. Locally, at any $p_{z}\,\in\, \mathrm{Gr}(r,\,n)$, we have
\begin{align*}
\mu_{\mathrm{Gr}}^{\mathrm{FS}}(p_{z})\,:=\,-\frac{\sqrt{-1}}{2\pi}\partial\overline{\partial}\log\|\mathfrak{s}\|_{\calH}^{2}
(p_{z})\, ,
\end{align*}
where $\mathfrak{s}$ is any holomorphic global section of $\calH$. 

Using equation \eqref{H-metric1}, we arrive at
\begin{align*}
\mu_{\mathrm{Gr}}^{\mathrm{FS}}(p_{z})=-\frac{\sqrt{-1}}{2\pi}\partial\overline{\partial}
\log\big(\mathfrak{s}^{\ast}(\mathrm{Id}_{r_{k}}+M_{z}^{\ast}M_{z})\mathfrak{s}\big)\notag\\
=\, -\frac{\sqrt{-1}}{2\pi}\partial\overline{\partial}\log\left(\mathrm{det}\big(\mathfrak{s}^{\ast}\big(\mathrm{Id}_{r_{k}}+M_{z}^{\ast}M_{z} \big)\mathfrak{s}\big)\right)=
-\frac{\sqrt{-1}}{2\pi}\partial\overline{\partial}\log\left(\mathrm{det}(\widetilde{M}_{z}^{\ast}\widetilde{M}_{z})\right).
\end{align*}

\begin{prop}\label{prop2}
At any point $z\,:=\,(z_1,\,\cdots,\,z_{d})\,\in\, \symxd$, the pulled back metric $$\mu_{\L,\symxd}^{\mathrm{FS,k}}
\,:=\,(\varphi_{\L}^{k})^{\ast}(\mu_{\mathrm{FS}})$$ on $\symxd$, where
$\varphi_{\L}^{k}$ is constructed in \eqref{mvp}, is given by the following formula:
\begin{align*}
\mu_{\L,\symxd}^{\mathrm{FS},k}(z)=-\frac{\sqrt{-1}}{2\pi}\sum_{i=1}^{d}\partial_{z_{i}}\partial_{\overline{z}_{i}}
\log(\bkld(z_{i}))\, ,
\end{align*} 
where $\bkld(z_{i})$ is the Bergman kernel associated to the line bundle
$\L^{\otimes k}\otimes {\mathcal O}_{X}(-\sum_{\ell=1}^d z_\ell)$.
\end{prop} 

\begin{proof}
Denote by $D$ the divisor $z_1+\ldots+z_d$. Let $\lbrace s_{1},\,\cdots,\, 
s_{r_{k}}\rbrace$ be a basis, of the above type, of the vector space $H^{0}(X,\,\L^{\otimes k}\otimes
{\mathcal O}_{X}(-D))$. Then, from the above identifications, where each column of the matrix corresponding 
to the point $p_{z}\,\in\, \mathrm{Gr}(r_{k},\,n_{k})$ is obtained by identifying each of the section $s_{i}$
with vectors in $\C^{n_{k}}$, we have 
\begin{align*}
(\varphi_{\L}^{k})^{\ast}\bigg(\partial\overline{\partial}\log({\rm{det}}(\widetilde{M}_{z}^{\ast}\widetilde{M}_{z}))\bigg)
\,=\,\sum_{i=1}^{d}\partial\overline{\partial}\log\bigg(\sum_{j=1}^{r_{k}}\|s_{j}\|_{k}^{2}(z)\bigg)_{\big\vert_{z=z_{i}}}
\,=\,\sum_{i=1}^{d}\partial_{z_{i}}\partial_{\overline{z}_{i}}\log (\bkld(z_{i}))\, ,
\end{align*}
which completes the proof.
\end{proof}

Now let $\L\,=\,\O$, and fix $d$. Let $\mathcal{B}_{\O,D}^{k}$ denote the Bergman kernel associated to the line bundle
$\Ok\otimes {\mathcal O}_{X}(-D)$, where $D$ is an effective divisor on $X$ of degree $d$. Then, for $k\,\gg\, 0$ and any
$z\,\in\, X$, from estimate \eqref{bkest0} we have
\begin{equation}\label{rem1}
\mathcal{B}_{\O,D}^{k}(z)\,=\,\bkx(z)+O\big(k^{-1}\big)\,=\,\frac{1}{2\pi}+O\big(k^{-1}\big)\, .
\end{equation}
Furthermore, let $\varphi_{\O}^{k}$ be the embedding, as in \eqref{mvp}, of $\symxd$ for the line bundle $\Ok$.
Let $\mu_{\symxd}^{\mathrm{FS},k}\,:=\, (\varphi_{\O}^{k})^{\ast} (\mu_{\mathrm{FS}})$ be the pull-back of the Fubini-Study metric on $\mathrm{Gr}(r_{k},\,n_{k})$. Let $\mu_{\symxd,\mathrm{vol}}^{\mathrm{FS},k}$ denote the volume form associated to $\mu_{\symxd}^{\mathrm{FS},k}$. Then, for $k\gg 0$, the following theorem gives an estimate of the ratio of volume forms $\mu_{\symxd,\mathrm{vol}}^{\mathrm{FS},k}\slash\hypxdvol$.

\begin{thm}\label{thm2}
For $k\,\gg\, 0$, and any $z\,:=\,(z_1,\,\cdots,\,z_d)\,\in\,\symxd$, the following estimate holds:
\begin{align*}
\bigg|\frac{\mu_{\symxd,\mathrm{vol}}^{\mathrm{FS},k}(z)}{\hypxdvol(z)}\bigg|
\,\leq \,\prod_{i=1}^{d}\Bigg(\frac{k^{2}}{\pi}\cdot\frac{\cx}{\bkx(z_{i})}\cdot\bigg(\frac{4\cx}{\bkx(z_{i})}+5+\frac{1}{2k}\bigg)+\frac{k}{2\pi}\Bigg)+o_{z}(k)\, ,
\end{align*}
where $o_{z}(k)$ represents a smooth function on $\symxd$ for each $k$, which for any $z\in\symxd$, goes to zero as $k$ tends to infinity.
\end{thm}

\begin{proof}
For $k\,\gg\, 0$, and any $z\,:=\,(z_1,\,\cdots,\,z_d)\in\symxd$, from \eqref{rem1} and Proposition \ref{prop2} we have 
\begin{align*}
\mu_{\symxd}^{\mathrm{FS},k}(z)\,=\,-\frac{\sqrt{-1}}{2\pi}\sum_{i=1}^{d}\partial_{z_{i}}\partial_{\overline{z}_{i}}
\log (\mathcal{B}_{\O,D}^{k}(z_{i}))\,=\,\sum_{i=1}^{d}\muberkx(z_i)+\widehat{o}_{z}(k)\, ,
\end{align*}
where $\widehat{o}_{z}(k)$ represents a smooth $(1,\,1)$-form on $\symxd$, which for any $z\in\symxd$, goes
to zero as $k$ tends to infinity. The
theorem now follows from Theorem \ref{thm1}. 
\end{proof}

\begin{cor}\label{cor2}
With notation as above, for any $z:=(z_1,\cdots,z_d)\in\symxd$, the following estimate holds:
\begin{align*}
\lim_{k\rightarrow \infty}\frac{1}{k^{2d}}\bigg|\frac{\mu_{\symxd,\mathrm{vol}}^{\mathrm{FS},k}(z)}{\hypxdvol(z)}\bigg|
\,\leq\, \bigg(\frac{26}{\pi}\bigg)^{d}\, .
\end{align*}
\end{cor}

\begin{proof}
This follows from Theorem \ref{thm2} and Corollary \ref{cor1}.
\end{proof}

\section*{Acknowledgements}

We thank the referee for helpful comments.
Both the authors thank ICTS Bangalore for their hospitality, where a major part of the work was carried out. The first-named
author wishes to thank TIFR Mumbai for their hospitality, where the work was completed. The first-named author also
acknowledges the support of INSPIRE research grant DST/INSPIRE/04/2015/002263, and the second-named author is supported by a
J. C. Bose Fellowship.


\begin{thebibliography}{AMM16}

\bibitem[AM1]{am}
A. Aryasomayajula and P. Majumder,
\newblock Off-diagonal estimates of the {Bergman} kernel on hyperbolic
{Riemann} surfaces of finite volume.
\newblock {\em Proceedings of AMS}, \textbf{146} (2018), 4009--4020.

\bibitem[AM2]{am2}
A. Aryasomayajula and P. Majumder,
\newblock Off-diagonal estimates of the {Bergman} kernel on hyperbolic
{Riemann} surfaces of finite volume-{II}.
\newblock {\em To appear in Annales Mathematiques Toulouse}; arXiv:1808.04646.

\bibitem[ABMS]{uoh} A. Aryasomayajula, I. Biswas, A. S. Morye, and T. Sengupta, On the K\"ahler metrics over $\symxd$, 
\textit{J. Geom. Phys.} \textbf{110} (2016), 187--194.

\bibitem[ACGH]{ACGH} E. Arbarello, M. Cornalba, P. A. Griffiths and J. Harris, {\it Geometry of algebraic curves. Vol. I},
Grundlehren der Mathematischen Wissenschaften, 267. Springer-Verlag, New York, 1985.

\bibitem[AMM]{au-ma2}
H. Auvray, X. Ma, and G. Marinescu, {Bergman} kernels on punctured {Riemann} surfaces,
{\em C. R. Math. Acad. Sci. Paris} \textbf{354} (2016), 1018--1022.

\bibitem[Ba]{ballman} 
W. Ballman, \textit{Lectures on K\"ahler manifolds}, ESI Lectures in Mathematics and Physics, 
European Mathematical Society (EMS), Z\"urich, 2006.

\bibitem[BR1]{biswas1}
I. Biswas and N.~M. Rom\~ao,
\newblock Moduli of vortices and {G}rassmann manifolds.
\newblock {\em Comm. Math. Phys.} {\bf 320} (2013), 1--20.

\bibitem[BR2]{BR2} I. Biswas and N.~M. Rom\~ao, A no-go theorem for nonabelionic statistics in gauged linear
sigma-models, {\it Adv. Theor. Math. Phys.} {\bf 21} (2017), 901--920.

\bibitem[Bo]{bouche}
T. Bouche,
\newblock Asymptotic results for {H}ermitian line bundles over complex
manifolds: the heat kernel approach.
\newblock In {\em Higher-dimensional complex varieties ({T}rento, 1994)}, pages
67--81. de Gruyter, Berlin, 1996.

\bibitem[De1]{De1} J.-P. Demailly, On the cohomology of pseudoeffective line bundles. {\it Complex geometry and dynamics}, 51--99,
Abel Symp., 10, Springer, Cham, 2015.

\bibitem[De2]{De2} J.-P. Demailly, Applications of pluripotential theory to algebraic geometry. {\it Pluripotential theory}, 
143--263, Lecture Notes in Math., 2075, Fond. CIME/CIME Found. Subser., Springer, Heidelberg, 2013.

\bibitem[Fr]{frietag} E. Freitag, {\it Hilbert Modular Forms}, Springer-Verlag, Berlin, 1990.

\bibitem[GH]{griffiths} P. Griffiths and J. Harris, \textit{Principles of algebraic geometry}, Pure and Applied
Mathematics, Wiley-Interscience, New York, 1978.

\bibitem[Ke]{Ke} G. Kempf, Toward the inversion of abelian integrals. I,
\textit{Ann. of Math.} \textbf{110} (1979), 243--273.

\bibitem[MM]{ma} X. Ma and G. Marinescu, {\it Holomorphic Morse inequalities and Bergman kernels},
 Progress in Mathematics, 254. Birkh\"auser Verlag, Basel, 2007. .

\bibitem[Mac]{mcdon} I. G. Macdonald, Symmetric products of an algebraic curve, \textit{Topology} \textbf{1} (1962), 
319--343.

\bibitem[Man]{Ma} N. S. Manton, One-vortex moduli space and Ricci flow, {\it Jour. Geom. Phys.} {\bf 58} (2008),
1772--1783.

\bibitem[MN]{MN} N. S. Manton and S. M. Nasir, Volume of vortex moduli spaces,
{\it Comm. Math. Phys.} {\bf 199} (1999), 591--604. 

\bibitem[Pe]{Pe} T. Perutz, Symplectic fibrations and the abelian vortex equations, {\it Comm. Math. Phys.}
{\bf 278} (2008), 289--306.

\bibitem[SZ]{SZ} B. Shiffman and S. Zelditch, Distribution of zeros of random and quantum chaotic sections of positive line 
bundles, {\it Comm. Math. Phys.} {\bf 200} (1999), 661--683.

\bibitem[Ti]{Ti} G. Tian, On a set of polarized Kahler metrics on algebraic manifolds,
{\it Jour. Diff. Geom.} {\bf 32} (1990), 99--130.

\bibitem[Ze]{Ze} S. Zelditch, Szeg\"o kernels and a theorem of Tian,
{\it Internat. Math. Res. Not.} 1998, no. 6, 317--331. 

\end{thebibliography}
\end{document}